\documentclass{article}

\usepackage{a4wide}
\usepackage{graphicx}
\usepackage{epsfig}
\usepackage{epstopdf}
\usepackage{amsmath,amsthm,amssymb}
\usepackage{url}

\newtheorem{theorem}{Theorem}[section]

\begin{document}

\title{Stochastic SICA Epidemic Model with Jump L\'{e}vy Processes\thanks{This is a 
preprint whose final form is published by Elsevier in the book 
'Mathematical Analysis of Infectious Diseases', 1st Edition -- June 1, 2022, 
ISBN: 9780323905046.}}

\author{Houssine Zine$^a$\\
\texttt{zinehoussine@ua.pt}
\and
Jaouad Danane$^b$\\
\texttt{jaouaddanane@gmail.com}
\and
Delfim F. M. Torres$^a$\thanks{Corresponding author: delfim@ua.pt}\\ 
\texttt{delfim@ua.pt}}


\date{$^a$\mbox{Center for Research and Development in Mathematics and Applications (CIDMA)},
Department of Mathematics, University of Aveiro, 3810-193 Aveiro, Portugal\\[0.3cm]
$^b$Laboratory of Systems Modelization and Analysis for Decision Support,\\ 
National School of Applied Sciences, Hassan First University, Berrechid, Morocco}

\maketitle


\begin{abstract}
We propose and study a shifted SICA epidemic model, extending the one of Silva and Torres (2017)
to the stochastic setting driven by both Brownian motion processes and jump L\'evy noise.
L\'evy noise perturbations are usually ignored by existing works of mathematical modelling 
in epidemiology, but its incorporation into the SICA epidemic model is worth to consider 
because of the presence of strong fluctuations in HIV/AIDS dynamics, often leading 
to the emergence of a number of discontinuities in the processes under investigation. 
Our work is organised as follows: (i) we begin by presenting our model, by clearly justifying  
its used form, namely the component related to the L\'evy noise; (ii) we prove existence and
uniqueness of a global positive solution by constructing a suitable stopping time; 
(iii) under some assumptions, we show extinction of HIV/AIDS; (iv) we obtain sufficient 
conditions assuring persistence of HIV/AIDS; (v) we illustrate our mathematical results
through numerical simulations.

\bigskip

\end{abstract}


\section{Introduction}

Human immunodeficiency virus (HIV) is known as a pathogen causing
the acquired immunodeficiency syndrome (AIDS), which is the
end-stage of the infection. After that, the immune system fails to
play its life-sustaining role \cite{bla,wei}.
On the other hand, according to the World Health Organization \cite{who}, 
$36.7$ million people live with HIV, $1.8$ million people become newly 
infected with HIV, and more than $1$ million individuals die annually. Based 
on these alarming statistics, HIV becomes a major global public health issue.
Mathematical modelling of HIV viral dynamics is a powerful tool for predicting 
the evolution of this disease \cite{Allali,smith,kor,now,kuang}.

On the other hand, stochastic quantification of several real life phenomena 
has been much helpful in understanding the random nature of their incidence 
or occurrence. This also helps in finding solutions to such problems, 
arising either in form of minimization of their undesirability 
or maximization of their rewards. Besides, the infectious diseases are exposed 
to randomness and uncertainty in terms of normal infection progress. 
Therefore, stochastic modelling is more appropriate comparing 
to deterministic, in particular considering the fact that 
stochastic systems do not only take into account the variable mean 
but also the standard deviation behaviour surrounding it. Moreover, 
the deterministic systems generate similar results for similar
initial fixed values, while the stochastic ones 
can give different predicted results. 
Several stochastic infectious models, describing the effect 
of white noise on viral dynamics, have been published
\cite{R2,Mahrouf2016,Pitchaimani2018,Rajaji2017,SIQR2019}.

In this paper, based on \cite{Silva16}, we propose and analyse 
a mathematical model for the transmission dynamics of HIV and AIDS. 
Our aim is to show the effect of the L\'{e}vy jump in the dynamics 
of the population. The L\'{e}vy noise is used to describe the 
contingency and the outburst. Precisely, we propose the 
following stochastic model driven jointly by white 
and L\'{e}vy noises:
\begin{equation}
\label{sy1}
\begin{cases}
dS(t)=\left( \Lambda-\beta I(t)S(t)-\mu S(t)\right)dt-\sigma I(t)S(t)dW_t
-\displaystyle \int_UJ(u)I(t-)S(t-)\check{N}(dt,du),\\
dI(t)=\left(\beta I(t)S(t)-(\rho+\phi+\mu)I(t)+\alpha A(t)
+\omega C(t)\right)dt +\sigma I(t)S(t)dW_t\\
\qquad\qquad +\displaystyle \int_UJ(u)I(t-)S(t-)\check{N}(dt,du),\\
dC(t)=\left( \phi I(t)-(\omega +\mu)C(t)\right)dt,\\
dA(t)=\left( \rho I(t)-(\alpha+\mu+d)A(t)\right)dt,
\end{cases}
\end{equation}
where $W_t$ is a standard Brownian motion with intensity $\sigma$ 
defined on a complete filtered probability space 
$\left( \Omega,\mathcal{F},(\mathcal{F}_t)_{t\geq 0},\mathbb{P}\right)$ 
with filtration $(\mathcal{F}_t)_{t\geq 0}$ satisfying the usual conditions; 
$S(t-)$ and $I(t-)$ denote the left limits of $S(t)$ and $I(t)$, respectively; 
$N(dt,du)$ is a Poisson counting measure with the stationary compensator 
$\nu(du)dt$ and $\tilde{N}(dt,du)=N(dt,du)-\nu(du)dt$, where $\nu$ 
is defined on a measurable subset $U$ of the non-negative half-line,  
with $\nu(U)< \infty$; and $J(u)$ represents the jumps intensity. 
Here, $S$ denotes the susceptible individuals; 
$I$ the HIV-infected individuals with no clinical symptoms of AIDS 
(the virus is living or developing in the individuals but 
without producing symptoms or only mild ones) 
but able to transmit HIV to other individuals; 
$C$ the HIV-infected individuals under ART treatment 
(the so called chronic stage) with a viral load remaining low; 
and $A$ the HIV-infected individuals with AIDS clinical symptoms. 
The meaning of the parameters of the SICA model \eqref{sy1} 
are given in Table~\ref{tab1}. 

\begin{table}
\caption{Parameters of the suggested stochastic 
SICA model \eqref{sy1} and their meaning.}
\centering
\begin{tabular}{|c | c |} \hline
Parameters &    Meaning  \\ \hline \hline
$\Lambda$ &  Recruitment rate\\ \hline
$\mu$  &  Natural death rate \\ \hline
$\beta$ & The transmission rate\\ \hline
$\phi$ & HIV  treatment rate for $I$ individuals\\ \hline
$\rho$ & Default treatment rate for $I$ individuals\\ \hline
$\alpha$ & AIDS treatment rate\\ \hline
$\omega$ & Default treatment rate for $C$ individuals\\ \hline
$d$ & AIDS induced death rate\\ \hline
\end{tabular}
\label{tab1}
\end{table}


\section{Existence and uniqueness of a global positive solution}
\label{sec2}

Let us define 
$$
\Omega :=\left\{(S;I;C;A)\in \mathbb{R}_+^4 : \frac{\Lambda}{\mu+d}
\leq S+I+C+A\leq \frac{\Lambda}{\mu}\right\}.
$$
Along the text, we assume that the following hypothesis 
on the jumps intensity $J(u)$ holds:
\begin{itemize}
\item[$(H)$] $J$ is a bounded function and 
$0<J(u)\leq \displaystyle \frac{\mu}{\Lambda}$, $u\in U$.
\end{itemize}
Moreover, we abbreviate ``almost surely'' as $a.s.$

We begin by proving existence of a unique global positive solution  
of system \eqref{sy1} for any given initial data in $\Omega$.

\begin{theorem}
If the given initial data $(S(0);I(0);C(0);A(0))$ belongs to $\Omega$,
then there exists $a.s.$ a unique global positive solution 
$(S(t);I(t);C(t);A(t))$ of system \eqref{sy1} in $\Omega$ 
for every $t\geq 0$. Moreover,
\begin{gather*}
\limsup_{t\rightarrow \infty} S(t)
\leq \dfrac{\Lambda}{\mu} \text{ a.s.}, 
\qquad\qquad \liminf_{t\rightarrow \infty} S(t)
\geq \dfrac{\Lambda}{\mu+d} \text{ a.s.},\\
\limsup_{t\rightarrow \infty} I(t)
\leq \dfrac{\Lambda}{\mu} \text{ a.s.}, 
\qquad\qquad \liminf_{t\rightarrow \infty} I(t)
\geq \dfrac{\Lambda}{\mu+d} \text{ a.s.},\\
\limsup_{t\rightarrow \infty} C(t)
\leq \dfrac{\Lambda}{\mu} \text{ a.s.}, 
\qquad\qquad \liminf_{t\rightarrow \infty} C(t)
\geq \dfrac{\Lambda}{\mu+d} \text{ a.s.},\\
\limsup_{t\rightarrow \infty} A(t)
\leq \dfrac{\Lambda}{\mu} \text{ a.s.}, 
\qquad\qquad \liminf_{t\rightarrow \infty} A(t)
\geq \dfrac{\Lambda}{\mu+d} \text{ a.s.}
\end{gather*} 
\end{theorem}

\begin{proof}
Given initial data $(S(0);I(0);C(0);A(0))\in \Omega $,
the local lipschitzianity of the drift and the diffusion 
enable us to confirm existence and uniqueness of a local solution  
$(S(t);I(t);C(t);A(t))$ in $\Omega $ for $t\in [0,\tau_e)$, 
where $\tau_e$ is the explosion time. To prove that such 
solution is global, we define the stopping time
$$
\tau=\left\{t\in  [0,\tau_e) : S(t)\leq 0, I(t)\leq 0, C(t)\leq 0, A(t)\leq 0\right\}.
$$
Assuming that $\tau_e< \infty$, we have $\tau\leq \tau_e$ and 
there exist $T>0$ and $\epsilon>0$ such that $P(\tau\leq T)>\epsilon$.
Let us now consider the following function $V$ on $\mathbb{R}_+^4$:
$V(x,y,z,t)=\log(xyzt)$. Using It\^{o}'s formula, we get:
\begin{multline*}
dV(t,X(t))=LV(t,X(t))dt+\partial_xV(t,X(t)) \cdot b(t,X(t))dW_t\\
+\int_U\left( V(X(t-)+J(u))- V(X(t-))\right)\check{N}(dt,du),
\end{multline*}
where $b$ is the drift coefficient, that is, in abbreviation,
$$
dV=LVdt+\left( \frac{\Lambda}{S}-\beta I+\beta S-\rho-\phi-2\mu+\alpha\frac{A}{I}
+\omega\frac{C}{I}\right)dW_t+\int_U\log(1-JI)(1+JS)\check{N}(dt,du),
$$
where $L$ denotes the differential operator. We have
\begin{multline*}
LV=(\Lambda-\beta IS-\mu S)\frac{1}{S}+(\beta IS-(\rho+\phi+\mu )I+\alpha A
+\omega C)\frac{1}{I}-\frac{\sigma^2I^2}{2}-\frac{\sigma^2S^2}{2}\\
+\int_U [\log(1-JI)(1+JS)-JI+JS] \nu(du)
\end{multline*}
and, noting that from our assumption $(H)$ one has $1-JI>0$, it follows that
$$
LV\geq -\frac{\beta \Lambda}{\mu}-2 \mu -\rho -\phi
- \frac{\sigma^2 \Lambda^2}{\mu^2}+\int_U [\log(1-JI)+JI] \nu(du)+\int_U [\log(1+JS)-JS] \nu(du):=K.
$$
Observe that $x \longmapsto \log(1+x)-x$ and 
$x \longmapsto \log(1-x)+x $ are non-positive functions.
Therefore,
\begin{multline}
\label{eq:last1}
dV\geq K dt+\left( \frac{\Lambda}{S}-\beta I+\beta S-\rho-\phi-2\mu+\alpha\frac{A}{I}
+\omega\frac{C}{I}\right)dW_t\\
+\int_U\log(1-JI)(1+JS)\check{N}(dt,du).
\end{multline}
Integrating \eqref{eq:last1} from $0$ to $t$, we get
\begin{equation*}
\begin{split}
V(S(t),I(t),C(t),A(t)) 
&\geq V(S(0),I(0),C(0),A(0))+K(t)\\
&\quad +\int_0^{t} \int_U\left( \frac{\Lambda}{S}-\beta I+\beta S
-\rho-\phi-2\mu+\alpha\frac{A}{I}+\omega\frac{C}{I}\right)dW_s\\
&\quad +\int_0^{t}\int_U\log(1-JI)(1+JS)\check{N}(ds,du).
\end{split}
\end{equation*}
Because of the continuity of the state variables, some components of 
$(S(\tau),I(\tau),C(\tau),A(\tau))$ are equal to $0$. Thus,  
$\lim_{t \rightarrow \tau}V(\tau)=-\infty$.
Letting $t \rightarrow \tau $, we deduce that
\begin{multline*}
-\infty \geq V(S(0),I(0),C(0),A(0))+K(t)
+\int_0^{t}\int_U\left( \frac{\Lambda}{S}-\beta I+\beta S-\rho-\phi-2\mu
+\alpha\frac{A}{I}+\omega\frac{C}{I}\right)dW_s\\
+\int_0^{t}\int_U\log(1-JI)(1+JS)\check{N}(ds,du)>\infty,
\end{multline*}
which contradicts our assumption. 
It remains to show the boundedness of the solution. 
Summing up the equations from system \eqref{sy1} gives that
$$
\dfrac{dN(t)}{dt}= \Lambda-\mu N(t)-dA(t),
$$
and upper and lower bounds are given by
$$ 
\Lambda-(\mu+d) N(t)\leq \dfrac{dN(t)}{dt}\leq \Lambda-\mu N(t),
$$
where $N(t)=S(t)+I(t)+C(t)+A(t)$. So,
\begin{align*}
e^{\mu t}\dfrac{dN(t)}{dt}
&\leq e^{\mu t}\left( \Lambda-\mu N(t)\right),\\
\int_0^t e^{\mu s}\dfrac{dN(s)}{ds}ds
&\leq \int_0^t e^{\mu s}\left( \Lambda-\mu N(s)\right)ds,\\
e^{\mu t} N(t)
&\leq \dfrac{\Lambda}{\mu}(e^{\mu t}-1)+N(0),\\
N(t)&\leq \dfrac{\Lambda}{\mu}(1-e^{-\mu t})+N(0) e^{-\mu t},
\end{align*}
and 
$$
\limsup_{t\rightarrow \infty} N(t) \leq \dfrac{\Lambda}{\mu} \text{ a.s.}
$$
Adopting the same technique, we also arrive to 
$\liminf_{t\rightarrow \infty} N(t)\geq \dfrac{\Lambda}{\mu+d} \text{ a.s.}$,
which confirms the intended boundedness.
\end{proof}


\section{Extinction}
\label{sec3}

We now provide a sufficient condition for the extinction of $I(t)$.

\begin{theorem}
If 
$$
\frac{\beta^2}{2 \sigma^2}< (\rho+\phi+\mu)+(\alpha+\omega)\frac{\Lambda}{\mu},
$$ 
then $I(t)\rightarrow 0$ a.s. when $t\rightarrow +\infty$.
\end{theorem}

\begin{proof}
Let $V(I)=\log(I)$. Using It\^{o}'s formula corresponding to the Poissonian process, we get
\begin{align*}
dV(t,X(t)) & = LV(t,X(t))dt+\partial_xV(t,X(t))\cdot \sigma I(t)S(t) dW_t\\
&\quad +\int_U\left( V(X(t-)+J(u))- V(X(t-))\right)\check{N}(dt,du)\\
& =  LV(t,X(t))dt+\sigma S(t) dW_t+\int_U\left(\log(1+JS)\right)\check{N}(dt,du),
\end{align*}
where
$$
LV=\left(\beta IS-(\rho+\phi+\mu)I+\alpha A+\omega C\right)
\frac{1}{I} -\frac{\sigma^2S^2}{2}+\int_U \log(1+JS)-JS] \nu(du),
$$
which implies that
$$
LV \leq \frac{\beta^2}{2 \sigma^2}-(\rho+\phi+\mu)
+(\alpha+\omega)\frac{\Lambda}{\mu}
$$
and
$$
dV \leq \left(\frac{\beta^2}{2 \sigma^2}-(\rho+\phi+\mu)
+(\alpha+\omega)\frac{\Lambda}{\mu}\right)dt+\sigma S(t)dW_t
+\int_U\log(1+JS)\check{N}(dt,du).
$$
Integrating from $0$ to $t$ and dividing by $t$ on both sides, we have
\begin{align*}
\dfrac{\log(V(t)}{t}
&\leq \dfrac{\log(I_0)}{t}+\dfrac{1}{t}
\int_0^t \left(\frac{\beta^2}{2 \sigma^2}-(\rho+\phi+\mu)
+(\alpha+\omega)\frac{\Lambda}{\mu}\right)ds\\
&\quad +\dfrac{1}{t}\int_0^t \sigma S(s)dB(s)
+\dfrac{1}{t}\int_0^t\int_U\log(1+JS)\check{N}(ds,du)\\
&\leq \dfrac{\log(I_0)}{t}+\frac{\beta^2}{2 \sigma^2}-(\rho+\phi+\mu)
+(\alpha+\omega)\frac{\Lambda}{\mu}+\dfrac{\sigma}{t}\int_0^t S(s)dB(s)\\
&\quad +\dfrac{1}{t}\int_0^t\int_U\log(1+JS)\check{N}(ds,du).
\end{align*}
Put
$$
M_t=\int_0^t \sigma S(s)dB_s,
$$
so that
\begin{align*}
\limsup_{t\rightarrow+\infty} \dfrac{<M_t,M_t>}{t}
=\limsup_{t\rightarrow+\infty} \dfrac{\sigma^2}{t}
\int_0^t u^2(s)ds
\leq \sigma^2\left( \dfrac{\Lambda}{\mu}\right)^2<\infty.
\end{align*}
Then, by using the strong law of large numbers theorem for martingales \cite{R7}, 
and the fact that the solution of the principal system is bounded, we get
\begin{align*}
\limsup_{t\rightarrow+\infty}\dfrac{M_t}{t}=0.
\end{align*}
Therefore,
\begin{equation*}
\limsup_{t\rightarrow+\infty}\dfrac{\log(I(t)}{t}
\leq  \frac{\beta^2}{2 \sigma^2}-(\rho+\phi+\mu)
+(\alpha+\omega)\frac{\Lambda}{\mu}.
\end{equation*}
Therefore, if $ \frac{\beta^2}{2 \sigma^2}
< (\rho+\phi+\mu)+(\alpha+\omega)\frac{\Lambda}{\mu}$, 
then $I(t)\rightarrow 0$ a.s. when $t\rightarrow +\infty$. 
\end{proof}


\section{Persistence in the mean}
\label{sec4}

In this section, we shall investigate the persistence property of 
$S(t)$, $I(t)$, $C(t)$, and $A(t)$ in the mean. This means that 
$$
\liminf_{t\rightarrow \infty}\dfrac{1}{t}\int_0^t x(s)ds>0,
$$
where $x(t)\in \{S(t),I(t),C(t),A(t)\}$.
For convenience, we define the following notation:
$$
<x(t)>:=\dfrac{1}{t}\int_0^t x(s)ds.
$$

\begin{theorem}
Let $(S(t), I(t),C(t),A(t))$ be a solution of system \eqref{sy1} 
with an arbitrary initial value $(S(0), I(0),C(0),A(0))$ 
in $\Omega$. If 
\begin{equation}
\beta \frac{\Lambda}{\mu+d}>(\rho+\phi+\mu)
+\frac{\sigma^2 \Lambda^2}{2 \mu^2},
\end{equation}
then
\begin{equation*}
\liminf_{t\rightarrow \infty}<I(t)> \  
\geq \frac{1}{(\rho+\phi+\mu)}\left(  \frac{\beta \Lambda}{\mu+d}
-(\rho+\phi+\mu)-\frac{\sigma^2 \Lambda^2}{2 \mu^2}\right) >0
\end{equation*}
and
\begin{equation*}
\liminf_{t\rightarrow \infty} <S(t)> \frac{\Lambda\mu}{\Lambda\beta + \mu^2} >0.
\end{equation*}
\end{theorem}

\begin{proof}
Using the first equation of system \eqref{sy1}, we have
\begin{align*}
\dfrac{S(t)-S(0)}{t} 
&=\frac{1}{t}\int_0^t(\Lambda-\beta IS-\mu S)ds
-\frac{1}{t}\int_0^t \sigma ISdW_s
-\frac{1}{t}\int_0^t \int_U J(u)IS \check{N}(ds,du)\\
&\geq \frac{1}{t}\int_0^t(\Lambda-\beta \frac{\Lambda}{\mu} S-\mu S)ds
-\frac{1}{t}\int_0^t \sigma ISdW_s-\frac{1}{t}\int_0^t \int_U J(u)IS \check{N}(ds,du)
\end{align*}
and
$$
\left(\frac{\Lambda\beta}{\mu}+\mu\right)<S> \ 
\geq \Lambda - \dfrac{S(t)-S(0)}{t}-\frac{1}{t}\int_0^t \sigma ISdW_s
-\frac{1}{t}\int_0^t \int_U J(u)IS \check{N}(ds,du).
$$
Using the strong law of large numbers for martingales 
and the boundedness of solution, we get
$$ 
\liminf_{t\rightarrow \infty} <S(t)> \ 
\geq \frac{\Lambda\mu}{\Lambda\beta + \mu^2}>0.
$$
Integrating the second equation of system \eqref{sy1} 
from $0$ to $t$, and dividing both sides by $t$, we obtain
\begin{align*}
\dfrac{I(t)-I(0)}{t} 
&=\frac{1}{t}\int_0^t(\left( \beta I(s)S(s)-(\rho+\phi+\mu)I(s)
+\alpha A(s)+\omega C(s)\right)ds\\
&\quad +\frac{1}{t}\int_0^t \sigma ISdW_s+
\frac{1}{t}\int_0^t \int_U J(u)IS \check{N}(ds,du)\\
& \geq -(\rho+\phi+\mu)<I(t)>+\frac{1}{t}\int_0^t \sigma ISdW_s
+\frac{1}{t}\int_0^t \int_U J(u)IS \check{N}(ds,du).
\end{align*}
Using again It\^{o}'s formula on function $V$ with 
$V(I)=\log(I)$, we get 
\begin{align*}
dV &=(\beta IS-(\rho+\phi+\mu)I+\alpha A+\omega C)\frac{1}{I}-\frac{\sigma^2S^2}{2}\\
&\quad +\int_U \log(1+JS)-JS] \nu(du))dt+\sigma S(t)dW_t+\int_U\log(1+JS)\check{N}(dt,du)\\
&\geq \left(\beta \frac{\Lambda}{\mu+d}
-(\rho+\phi+\mu)+\alpha \frac{\mu}{\mu+d}
+\omega \frac{\mu}{\mu+d}\right)
-\frac{\sigma^2S^2}{2}\\
&\quad +\int_U \log(1+JS)-JS] \nu(du))dt
+\sigma S(t)dW_t+\int_U\log(1+JS)\check{N}(dt,du)
\end{align*}
and
\begin{align*}
\frac{\log I(t)-\log I(0)}{t} 
& \geq \left(\beta \frac{\Lambda}{\mu+d}-(\rho+\phi+\mu)
+\alpha \frac{\mu}{\mu+d}+\omega \frac{\mu}{\mu+d}\right)
-\frac{\sigma^2 \Lambda^2}{2 \mu^2}\\
&\quad +\frac{1}{t} \int_U \log(1+JS)-JS] \nu(du))ds
+\frac{1}{t} \int_0^t \sigma S(s)dW_s\\
&\quad +\frac{1}{t} \int_0^t \int_U\log(1+JS)\check{N}(ds,du).
\end{align*}
By summing $\dfrac{I(t)-I(0)}{t}$ and $\dfrac{\log I(t)-\log I(0)}{t}$, 
applying the strong law of large numbers for martingales, 
and using the positivity and boundedness of the solution, we get
\begin{align*}
\liminf_{t\rightarrow \infty}<I(t)> 
& \geq \frac{1}{(\rho+\phi+\mu)}\left(  
\frac{\beta \Lambda}{\mu+d}-(\rho+\phi+\mu)
-\frac{\sigma^2 \Lambda^2}{2 \mu^2}\right) > 0.
\end{align*}
Consequently, the intended persistence 
in mean of $I(t)$ holds.
\end{proof}


\section{Numerical results}
\label{sec5}

This section is devoted to illustrate our  mathematical findings 
through numerical simulations. In the following examples, we apply the 
algorithm presented in \cite{Zou2014} to solve system  
\eqref{sy1} and we take the parameter values as given in Table~\ref{tabl1}.
\begin{table}[h!]
\caption{Parameters values used in the numerical simulations.} 
\label{tabl1}
\centering
\begin{tabular}{|c|c|c|} \hline \hline 
Parameters & Fig.~\ref{fig1} & Fig.~\ref{fig2} \\ \hline \hline 
$\Lambda$  & $10$     & $100$\\ 
$\mu$	   & $0.0125$ & $0.0013$\\ 
$\beta$    & $0.0001$ & $0.1$ \\
$\phi$     & $1$      & $1$\\
$\rho$     & $0.1$    & $0.1$\\
$\alpha$   & $0.33$   & $0.33$\\
$\omega$   & $0.09$   & $0.09$\\
$d$        & $1$      & $1$ \\ \hline
\end{tabular}
\end{table}
\begin{figure}[!t]
\includegraphics[width=14cm,height=7cm]{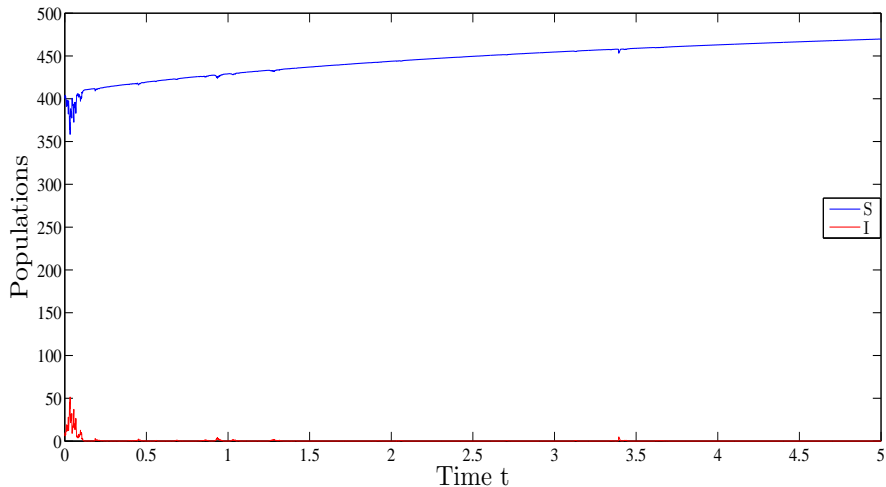} 
\caption{Susceptible and infected populations 
as functions of time in a situation of disease extinction.}
\label{fig1}
\end{figure}

Figure~\ref{fig1} shows the dynamics of susceptible 
and infected classes during the period of observation 
for the case of disease extinction. From Figure~\ref{fig1}, 
we clearly observe that the curve representing  
the infected population converges to $0$. It is worth 
to notice that, in this case, the susceptible individuals 
increase to reach their maximum, which means that the disease dies out. 
This is consistent with our theoretical findings of Section~\ref{sec3}
concerning extinction in the SICA model.

\begin{figure}[!t]
\includegraphics[width=14cm,height=7cm]{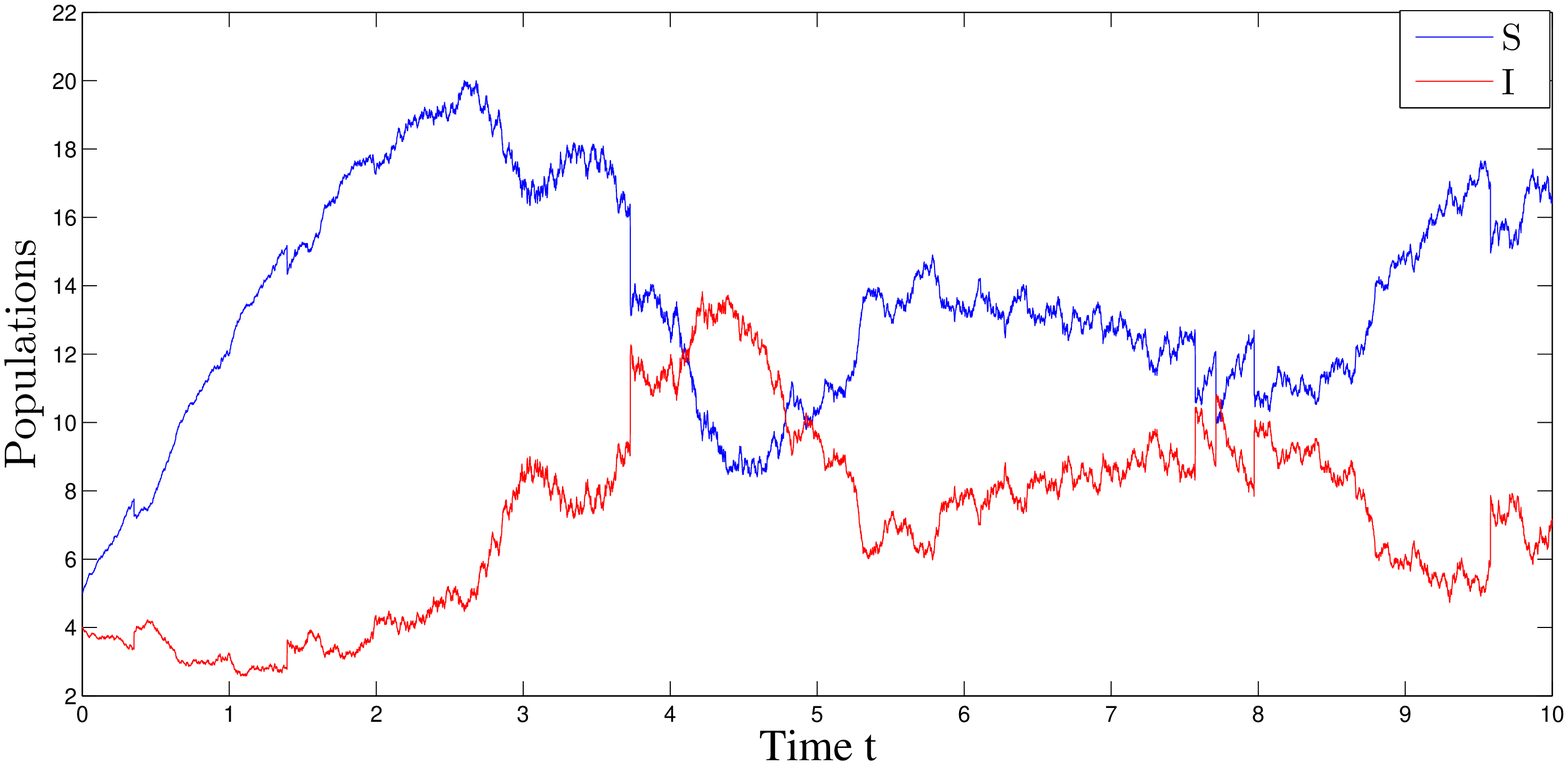} 
\caption{Susceptible and infected populations as functions 
of time in a situation of disease persistence.}
\label{fig2}
\end{figure}

The evolution of both susceptible and infected populations, 
as predicted by our L\'evy jumps model \eqref{sy1}, is 
illustrated in Fig.~\ref{fig2} in the case of 
disease persistence. We note that in this epidemic situation, 
all the four  SICA compartments, i.e., the susceptible, 
the infected, the HIV-infected individuals under ART treatment 
(the so called chronic stage) with a viral load remaining low, 
and the HIV-infected individuals with AIDS clinical symptoms, 
persist. This is consistent with our theoretical findings 
of Section~\ref{sec4} concerning persistence.


\section{Conclusion}
\label{sec6}

In this work, we have considered and extended the  
epidemic SICA model of Silva and Torres \cite{Silva16} 
to a new stochastic model driven by both white noise 
and L\'{e}vy noise. This allows to better describe 
the sudden social fluctuations. The new SICA model 
was studied theoretically and some numerical simulations 
were also performed, which not only support the proved mathematical 
results but also illustrate the asymptotic behaviour of the solution. 
Firstly, with the help of Lyapunov's analysis method,
we have proved existence and uniqueness of a solution.  
Secondly, we have demonstrated that the model is well-posed, both
mathematically and biologically, by establishing the boundedness 
of the solution, that is, $\limsup_{t\rightarrow \infty} N(t) 
\leq \dfrac{\Lambda}{\mu}$ a.s. and
$\liminf_{t\rightarrow \infty} N(t)\geq \dfrac{\Lambda}{\mu+d}$ a.s.,
as well as the positivity of the solution.
Thirdly, we have obtained an appropriate sufficient condition 
for extinction, showing that with an effective threshold of an eventual 
big magnitude of the volatility $\sigma$, 
$\frac{\beta^2}{2 \sigma^2}< (\rho+\phi+\mu)+(\alpha+\omega)\frac{\Lambda}{\mu}$, 
the eradication of the disease occurs. Fourthly, a novel and significant sufficient condition
$$
\liminf_{t\rightarrow \infty}<I(t)>  \ 
\geq \frac{1}{(\rho+\phi+\mu)}\left( \frac{\beta \Lambda}{\mu+d}
-(\rho+\phi+\mu)-\frac{\sigma^2 \Lambda^2}{2 \mu^2}\right) >0
$$ 
and  
$$
\liminf_{t\rightarrow \infty} <S(t)> \frac{\Lambda\mu}{\Lambda\beta + \mu^2} >0
$$
for persistence is obtained, which means that with an adopted 
small magnitude of volatility $\sigma$, the model is persistent in mean.
Lastly, some numerical simulations were implemented that confirm and illustrate 
our mathematical results, give some supplementary insights and eventually helps 
a decision maker to select a good strategy to control the disease by means of 
the increasing or decreasing of the intensity of volatility
or by taking into account and influencing the L\'{e}vy noise 
on the evolution of the variables of the system.


\section*{Acknowledgements}

H.Z. and D.F.M.T. were supported by FCT within project UIDB/04106/2020 (CIDMA).



\end{document}